\documentclass[12pt]{amsart}

\usepackage{amsfonts}
\usepackage{amsmath}
\usepackage{amssymb}

\usepackage{mathrsfs}
\usepackage[colorlinks]{hyperref}

\usepackage{epsfig}
\usepackage{graphicx}

\numberwithin{equation}{section}
\theoremstyle{plain}
\newtheorem{thm}{Theorem}[section]
\newtheorem{prop}[thm]{Proposition}
\newtheorem{cor}[thm]{Corollary}
\newtheorem{lemma}[thm]{Lemma}

\theoremstyle{definition}
\newtheorem{defi}[thm]{Definition}
\newtheorem{rem}[thm]{Remark}

\newcommand{\eps}{\varepsilon}

\begin{document}

\title[Wave propagation with irregular dissipation]
{Wave propagation with irregular dissipation and applications to acoustic problems and shallow waters}

\author[Juan Carlos Mu\~noz]{Juan Carlos Mu\~noz}
\address{
  Juan Carlos Mu\~noz:
  \endgraf
  Department of Mathematics
  \endgraf
  Universidad del Valle
  \endgraf
  Calle 13 Nro 100-00,  Cali
  \endgraf
  Colombia
  \endgraf
  {\it E-mail address} {\rm jcarlmz@yahoo.com}
  }

\author[Michael Ruzhansky]{Michael Ruzhansky}
\address{
  Michael Ruzhansky:
  \endgraf
  Department of Mathematics
  \endgraf
  Imperial College London
  \endgraf
  180 Queen's Gate, London, SW7 2AZ
  \endgraf
  United Kingdom
  \endgraf
  {\it E-mail address} {\rm m.ruzhansky@imperial.ac.uk}
  }
\author[Niyaz Tokmagambetov]{Niyaz Tokmagambetov}
\address{
  Niyaz Tokmagambetov:
  \endgraf
    al--Farabi Kazakh National University
  \endgraf
  71 al--Farabi ave., Almaty, 050040
  \endgraf
  Kazakhstan,
  \endgraf
   and
  \endgraf
    Department of Mathematics
  \endgraf
  Imperial College London
  \endgraf
  180 Queen's Gate, London, SW7 2AZ
  \endgraf
  United Kingdom
  \endgraf
  {\it E-mail address} {\rm n.tokmagambetov@imperial.ac.uk}
 }

\thanks{The authors were supported in parts by the EPSRC
grant EP/K039407/1 and by the Leverhulme Grant RPG-2014-02,
as well as by the MESRK grant 0773/GF4. Juan Carlos Mu\~noz was supported by Universidad del Valle (Colombia) and Colciencias under grant 1106-712-50006. No new data was collected or generated during the course of research. }

\date{\today}

\subjclass{42A85, 44A35.} \keywords{acoustic equation, shallow water, Cauchy problem, dissipative wave equation}

\begin{abstract}
In this paper we consider an acoustic problem of wave propagation through a discontinuous medium. The problem is reduced to the dissipative wave equation with distributional dissipation. We show that this problem has a so-called very weak solution, we analyse its properties and illustrate the theoretical results through some numerical simulations by approximating the solutions to the full dissipative model for a particular synthetic piecewise continuous medium. In particular, we discover numerically a very interesting phenomenon of the appearance of a new wave at the singular point.
For the acoustic problem this can be interpreted as an echo effect at the discontinuity interface of the medium.
\end{abstract}

\maketitle

\section{Introduction}

This work is devoted to the investigation of the 1D wave propagation through a medium with  positive piecewise regular density and wave speed functions.
For these non-smooth data we show that the problem has a so-called `very weak' solution. This notion has been introduced in \cite{GR15} in the analysis of second order hyperbolic equations, and in \cite{RT16} it was applied to show the well-posedness of the wave equations for the Landau Hamiltonian with irregular electro-magnetic fields. In this paper we use it to prove the well-posedness of the acoustic problem. Moreover, it allows us to derive the decay properties in time also in the situation when the medium has discontinuities.
Incidentally, the same model equation (see \eqref{EQ-AP:01-0}) appears also in the shallow water equations (see \eqref{EQ:Bo}) as a special case of the linear Boussinesq system, so the obtained results apply in that situation as well.

We start with a description of the physical problem.
Now, let $t$ denote the time and let $z$ be the Cartesian coordinate in the direction of the wave propagation. Let $\rho$ denote the density and $c$ the wave speed of the medium. Following the derivation in \cite{BB95}, we obtain the first order hyperbolic system with the constitutive equation and the momentum equation:
\begin{equation}\label{EQ: 1D-LHS-C&M eq-ns}
p_{t}+\rho c^{2}\omega_{z}=0, \,\,\,\,\,\,\,\,  \rho\omega_{t}+p_{z}=0,
\end{equation}
where $p=p(z, t)$ is the $z$--component of traction across surfaces $z=C$ ($C$ -- positive in compression), and $\omega=\omega(z,t)$ is the $z$--component of particle velocity at the point $(z, t)$. The first equation in \eqref{EQ: 1D-LHS-C&M eq-ns} is Hooke's law differentiated in $t$. By denoting $$x=\int_{0}^{z}\frac{1}{c(s)}ds,$$
we rewrite \eqref{EQ: 1D-LHS-C&M eq-ns} with respect to the impedance $\zeta(x)=\rho(z)c(z)$
(see \cite{BB83} for the detailed argument) and obtain:
\begin{equation}\label{EQ: 1D-LHS-C&M eq-ns-2}
p_{t}+\zeta\omega_{x}=0, \,\,\,\,\,\,\,\,  \zeta\omega_{t}+p_{x}=0.
\end{equation}
Putting together equations \eqref{EQ: 1D-LHS-C&M eq-ns-2}, we get the equation
\begin{equation}\label{EQ: 1D-LHS-C&M eq-ns-3}
\omega_{tt}-\frac{\zeta'(x)}{\zeta(x)}\omega_{x}-\omega_{xx}=0.
\end{equation}
Putting the initial conditions at $x=0$, the problem \eqref{EQ: 1D-LHS-C&M eq-ns-2}
was analysed in \cite{BB95} giving a rigorous estimate of the error occuring in making the O'Doherty--Anstey approximation, originally derived in 1971 in the context of acoustic wave propagation in the earth's crust \cite{ODA}. Furthermore, in \cite{BB95} it was shown that the down-going wave component $D=\zeta^{1/2} \omega + \zeta^{-1/2} p$ decays as $x\to\infty$ (i.e.  large propagation distance) for smooth positive functions $\rho$ and $c$, so that $\zeta$ is also smooth and positive. An extensive physical discussion of this kind of model equations was done in \cite{BB95}, \cite{BB83}, \cite{CF94}, \cite{LB96}, \cite{PS00}, \cite{AKPP}. Extension of the O'Doherty--Anstey approximation for weakly-dispersive, weakly nonlinear water waves propagating on the surface of a shallow channel with a random depth was developed in \cite{MN04}, \cite{GMN07}, and \cite{MN05}.

In this paper we are interested in the problem of existence of solutions of the model equation \eqref{EQ: 1D-LHS-C&M eq-ns-3} in the situation when the density $\rho$ and the wave speed $c$ of the medium are irregular. For example, we want to allow them to be discontinuous or, in general, have even less regularity. At the same time, it is natural, from the physical meaning of these functions, to continue assuming that they are positive:
$$
\rho>0,\; c>0 \textrm{ so that also } \zeta=\rho c>0.
$$
The first idea in our analysis is to observe that if we change the roles of $t$ and $x$, the equation \eqref{EQ: 1D-LHS-C&M eq-ns-3} takes the form of the dissipative wave equation. Indeed, swapping variables $x$ and $t$, and denoting $u:=\omega$ and $b:=\zeta$, the boundary value problem \eqref{EQ: 1D-LHS-C&M eq-ns-3} is reduced to the Cauchy problem for the one dimensional acoustic equation
\begin{equation}\label{EQ-AP:01-0}
\left\{ \begin{split}
\partial_{t}^{2}u(t,x)-\partial_{x}^{2}u(t,x)+\frac{b'(t)}{b(t)}\partial_{t}u(t,x)&=0, \; (t,x)\in [0, \infty)\times \mathbb R,\\
u(0,x)&=u_{0}(x), \; x\in \mathbb R, \\
\partial_{t}u(0,x)&=u_{1}(x), \; x\in \mathbb R,
\end{split}
\right.
\end{equation}
with positive $b>0$.
Now, even if $b$ is a regular function, the dissipation speed $\frac{b'(t)}{b(t)}$ has to be positive if we want to have the decay of solutions for large $t$. Indeed, this corresponds to the system losing energy (rather than gaining it from outside) and is a natural physical assumption. Since $b>0$, this means that we should also have $b'>0$. From the theory of distributions we know that these conditions would imply that $b$ and $b'$ are positive Radon measures, so that it is natural to assume that $b$ is a piecewise continuous (and hence also increasing) function.

Therefore, these will be the assumptions for our analysis, namely, we assume that the product $b=\zeta=\rho c$ of the density and the wave speed of the medium is an increasing piecewise continuous function. If they are smooth (or at least $C^{1}$), this is a natural physical assumption (see \cite{BB95}). Thus, the main novelty of this paper is that we relax the regularity assumption
\begin{equation}\label{EQ:as}
\textrm{\em requiring only that $b$ is piecewise continuous, positive and increasing}.
\end{equation}
While physically this is a very reasonable setting, mathematically we face several problems:

\begin{itemize}
\item since $b'$ would contain delta-functions at the discontinuity points of $b$, the coefficient $\frac{b'(t)}{b(t)}$ is not well-defined as a distribution;
\item moreover, if the data $u_{0}$ and $u_{1}$ are irregular, the dissipation term $\frac{b'(t)}{b(t)}\partial_{t}u(t,x)$ also does not make sense as a distribution in view of the celebrated Schwartz' impossibility result in \cite{S54} on multiplication of distributions.
\end{itemize}
Nevertheless, in this paper we analyse the equation \eqref{EQ-AP:01-0} or, more generally \eqref{EQ-AP:01}, under the assumption \eqref{EQ:as}. In particular, we show that it is well-posed in the sense of very weak solutions introduced in \cite{GR15}, and then also used in \cite{RT16} in another context. In particular, we show that

\begin{itemize}
\item if the Cauchy data $(u_0, u_1)$ is in the Sobolev spaces ${H}^{s+1}\times {H}^{s}$, $s\geq 0$, then the Cauchy problem \eqref{EQ-AP:01-0} has a very weak solution of order $s$; the very weak solution is unique in an appropriate sense;
\item if $b\in C^{1}$ we know that
the Cauchy problem \eqref{EQ-AP:01-0} also has a classical solution in
$C([0, \infty),{H}^{s+1}) \cap C^1([0, \infty),{H}^{s})$; in this case the very weak solution recaptures the classical solution;
\item we also give the above results for the Cauchy data $(u_0, u_1)$ in the Sobolev spaces ${H}^{s+1}\times {H}^{s}$ for any $s\in\mathbb R$;
\item under assumption \eqref{EQ:as}, the very weak solution is uniformly bounded in $L^{2}$ and may decay in $t$ depending on further properties of $b$.
\end{itemize}
The last property of the decay depending on further properties of $b$ is natural even for smooth functions $b$, see \cite{W04, W06,W07}. We note that if $b$ is regular, we could write $\frac{b'(t)}{b(t)}=(\log b(t))'$, which could be interpreted as a measure since $\log b(t)$ is a function of bounded variation. However, such a relation appears to be only formal since the quotient $\frac{b'(t)}{b(t)}$ is not well-defined even as a distribution. This also gives a novelty compared to the setting of \cite{GR15} since the coefficients of the wave equations there were assumed to be distributions. We also note that for functions $b$ more regular than $C^1$, there are many results available, for the well-posedness in Gevrey spaces and in spaces of ultradistributions, see e.g.
\cite{Colombini-deGiordi-Spagnolo-Pisa-1979, DS, KS} and references therein, to mention only very few.

\medskip

Let us mention that the equation \eqref{EQ-AP:01-0} also appears in the modelling of the non-dispersive water wave propagation in shallow water. Namely, the linear case of the Boussinesq systems takes the form
\begin{equation}\label{EQ:Bo}
\begin{split}
M(\xi) \eta_t + u_\xi = 0, \\
u_t + \eta_\xi = 0.
\end{split}
\end{equation}
Indeed, this is the non-dispersive case of the system analysed in \cite[(2.5)]{MN04}
(the case of $\beta=0$) to which we refer to further physical details. The function $M$ is determined by the channel depth and local wave speed, and is discontinuous in channels with sudden changes in the depth.

Again, with the change of variable
$$
x= \int_0^\xi \frac{1}{C_0(s) } ds,
$$
where $C_0(x)= \sqrt{1/M(x)},$  we arrive at the equation
$$
u_{tt} - \frac{C_0'(x)}{C_0(x)} u_x - u_{xx} = 0,
$$
which is the same model \eqref{EQ: 1D-LHS-C&M eq-ns-3}, where the coefficient $C_0$ (the local wave speed) has the role of the impedance.

\medskip
{\bf Numerics.} In Section \ref{Numerics} we make a numerical modelling of the problem with $b(t)$ having a jump discontinuity. We see that the approximating technique of very weak solutions does allow us to recover some physically expected behaviour such as the decay of solutions for large times, making such considerations mathematically rigorous. Moreover, we observe a very interesting phenomenon of the appearance of a new wave after the singular time travelling in the direction opposite to the main one.
{\em For the original acoustic problem this can be interpreted as an echo effect at the discontinuities of the medium.}
Such phenomenon is known in the presence of multiple characteristics in hyperbolic equations and is related to conical refraction, see e.g. \cite{MU79} or \cite{L93} for different descriptions. In that case the newly appearing wave is weaker, see e.g. \cite{KR07}, which is consistent with the observed pictures. The difference, however, is that in our case such behaviour is not related to multiple characteristics but to the singularity in the dissipation coefficient. Indeed, we analyse it further in Figures \ref{propagationU3}-\ref{propagationU5} by modelling a delta distribution as the Cauchy data. We see that the second wave is not more regular than the main one: it appears to be the second delta function (thus, singularity of comparable strength to the first one) but smaller in amplitude.

\medskip

In Section \ref{SEC:main} we discuss the known results in the case when $b$ is regular, and then formulate our results concerning very weak solutions for irregular $b$. We also consider the problem in higher dimensions, which is of independent mathematical interest. In Section \ref{SEC:pf1} we prove the results. In Section \ref{SEC:distr} we briefly discuss the case of distributional Cauchy data. Finally, in Section \ref{Numerics} we include some numerical experiments to illustrate the theoretical results derived in this paper for a particular synthetic piecewise continuous impedance function $\zeta$.

\section{Main Results}
\label{SEC:main}

In this work we can deal with the Cauchy problem \eqref{EQ-AP:01-0} in $\mathbb R^{n}$ since our approach works in any dimension $n\geq 1$. More precisely, we can consider
\begin{equation}\label{EQ-AP:01}
\left\{ \begin{split}
\partial_{t}^{2}u(t,x)-\Delta_{x} u(t,x)+\frac{b'(t)}{b(t)}\partial_{t}u(t,x)&=0, \; (t,x)\in [0, \infty)\times \mathbb R^{n},\\
u(0,x)&=u_{0}(x), \; x\in \mathbb R^{n}, \\
\partial_{t}u(0,x)&=u_{1}(x), \; x\in \mathbb R^{n},
\end{split}
\right.
\end{equation}
where $\Delta$ is the Laplace operator in $\mathbb R^{n}$. From the point of view of the acoustic problem the case $n=1$ is relevant, but mathematically the equation \eqref{EQ-AP:01} with irregular $b$ as in \eqref{EQ:as} is of independent interest.

We start by briefly recalling the notion of very weak solutions. For this, we do not need the assumption on the continuity of $b$ since it is available in a much more general context.

Thus, let $b$ be a positive distribution, i.e. there is a constant $b_{0}>0$ such that $b\geq b_{0}>0$. Here $b\geq b_{0}$ means that $b-b_{0}\geq0$, or $\langle b-b_{0}, \psi\rangle\geq0$ for all $\psi\in C^\infty_0(\mathbb R)$, $\psi\ge 0$.

As already mentioned, we will be using the notion of very weak solutions which was formulated for wave equations with spacially constant coefficients in \cite{GR15}, and applied to the wave equation for the Landau Hamiltonian in \cite{RT16}. We start by regularising the distributional coefficient $b$ with a suitable mollifier $\psi$ generating families of smooth functions $(b_{\eps})_\eps$, namely
\begin{equation}\label{EQ:regs}
b_{\eps}=b\ast\psi_{\omega(\eps)},
\end{equation}
where  $$\psi_{\omega(\eps)}(t)=\omega(\eps)^{-1}\psi(t/\omega(\eps))$$ and $\omega(\eps)$ is a positive function converging to $0$ as $\eps\to 0$ to be chosen later (sometimes we need a particular behaviour in $\eps$, see \cite{GR15} and \cite{RT16} but the situation for \eqref{EQ-AP:01} is simpler).
Here $\psi$ is a Friedrichs--mollifier, i.e.  $\psi\in C^\infty_0(\mathbb R)$, $\psi\ge 0$ and $\int\psi=1$. It follows that the net $(b_{\eps})_\eps$ is $C^\infty$-\emph{moderate}, in the sense that its $C^\infty$-seminorms can be estimated by a negative power of $\eps$. More precisely, following \cite{GR15}, we will use the notions of moderateness:
\begin{defi}
\label{def_mod_intro}
\leavevmode
\begin{itemize}
\item[(i)]
A net of functions $(f_\eps)_{\eps\in (0,1]}\subset C^\infty(\mathbb R)$
is said to be $C^\infty$-moderate if for all $K\Subset\mathbb R$ and for all $\alpha\in\mathbb N_{0}$ there exist $N=N_{\alpha}\in\mathbb N_{0}$ and $c=c_{\alpha}>0$ such that
\[
\sup_{t\in K}|\partial^\alpha f_\eps(t)|\le c\eps^{-N-\alpha}
\]
holds for all $\eps\in(0,1]$.
\item[(ii)] A net of functions $(u_\eps)_{\eps\in (0,1]}\subset C^{k}([0, \infty);{H}^{s-k})$ for all $k\in\mathbb N_0$ is said to be $C^\infty([0, \infty);{H}^{s})$-moderate if there exist $N\in\mathbb N_{0}$ and for all $k\in\mathbb N_{0}$ there is $c_k>0$  such that
\[
\|\partial_t^k u_\eps(t,\cdot)\|_{H^{s-k}}\le c_k\eps^{-N-k},
\]
for all $t\in[0, \infty)$ and $\eps\in(0,1]$.
\end{itemize}
\end{defi}
Here and in the sequel, the notation $K\Subset\mathbb R$ means that $K$ is
a compact set in $\mathbb R$.

We note that the conditions of moderateness are natural in the sense that regularisations of distributions are moderate, namely we can regard
\begin{equation}\label{EQ:incls}
\textrm{ compactly supported distributions } \mathcal{E}'(\mathbb R)\subset \{C^\infty \textrm{-moderate families}\}
\end{equation}
by the structure theorems for distributions.
Following and adapting \cite{GR15}, we now define the notion of a `very weak solution' for the Cauchy problem \eqref{EQ-AP:01}:

\begin{defi}
\label{def_vws}
Let $s\in\mathbb R$ and $(u_{0},u_{1})\in {H}^{s}\times {H}^{s-1}$. The net $(u_\eps)_\eps\in C^\infty([0, \infty);{H}^{s})$ is {\em a very weak solution of order $s$} of the
Cauchy problem \eqref{EQ-AP:01} if there exists
\begin{itemize}
\item[]
a $C^\infty$-moderate regularisation $b_{\eps}$ of the coefficient $b$,
\end{itemize}
such that $(u_\eps)_\eps$ solves the regularised problem
\begin{equation}\label{EQ: CPbb}
\left\{ \begin{split}
\partial_{t}^{2}u_{\eps}(t,x)-\Delta_{x}u_{\eps}(t,x)+\frac{b'_{\eps}(t)}{b_{\eps}(t)}\partial_{t}u_{\eps}(t,x)&=0, \; (t,x)\in [0, \infty)\times \mathbb R^{n},\\
u_{\eps}(0,x)&=u_{0}(x), \; x\in \mathbb R^{n}, \\
\partial_{t}u_{\eps}(0,x)&=u_{1}(x), \; x\in \mathbb R^{n},
\end{split}
\right.
\end{equation}
for all $\eps\in(0,1]$, and is $C^\infty([0, \infty);{H}^{s})$-moderate.
\end{defi}

In the sequel, the proofs of the main statements (and especially the decay of the very weak solutions) will depend on behaviour of $b'_{\eps}(t)/b_{\eps}(t)$ at $t\rightarrow\infty$. In order to put the problem in perspective and for our subsequent use, we consider several cases.
First, when for a fixed $\eps>0$, we have $t(b'_{\eps}(t)/b_{\eps}(t))\rightarrow\infty$ as $t\rightarrow\infty$ then we will use the following result following from \cite{W07}:

\begin{thm}\label{TH: W07-1}
Let $\eps>0$.
If $t(b'_{\eps}(t)/b_{\eps}(t))\rightarrow\infty$ as $t\rightarrow\infty$, then
the solution to \eqref{EQ: CPbb} and its derivatives satisfy the Matsumura-type estimates
\begin{equation}\label{EQ: W07.01}
\|\partial_{x}^{\alpha}u_{\eps}\|_{L^{2}}\lesssim\left(1+\int\limits_{0}^{t} \frac{b_{\eps}(\tau)}{b'_{\eps}(\tau)}d\tau\right)^{-|\alpha|/2}(\|u_{0}\|_{H^{s+|\alpha|}}+\|u_{1}\|_{H^{s+|\alpha|-1}})
\end{equation}
and
\begin{equation}\label{EQ: W07.02}
\|\partial_{t}\partial_{x}^{\alpha}u_{\eps}\|_{L^{2}}\lesssim\frac{b_{\eps}(t)}{b'_{\eps}(t)}\left(1+\int\limits_{0}^{t} \frac{b_{\eps}(\tau)}{b'_{\eps}(\tau)}d\tau\right)^{-|\alpha|/2-1}(\|u_{0}\|_{H^{s+|\alpha|+1}}+\|u_{1}\|_{H^{s+|\alpha|}})
\end{equation}
as $t\to\infty$ for any $s>0$ and for all $\alpha\in\mathbb N_{0}^{n}$.
\end{thm}
Theorem \ref{TH: W07-1} follows from \cite[Result 1]{W07} by setting $p=q=2$.
Integrals in \eqref{EQ: W07.01} and \eqref{EQ: W07.02} are well--defined since $b_{\eps}'$ is also a positive function which will be discussed below.

\vspace{3mm}

In the cases when $t(b'_{\eps}(t)/b_{\eps})=O(1)$ or $t(b'_{\eps}(t)/b_{\eps})=o(1)$ as $t\rightarrow\infty$ we will use the following results from \cite{W06}:

\begin{thm}\label{TH: W06-1}
Let $\eps>0$. Assume that $t(b'_{\eps}(t)/b_{\eps})=O(1)$ as $t\rightarrow\infty$.
Then the solution to \eqref{EQ: CPbb} satisfies the estimates
\begin{equation}\label{EQ: W06.01}
\|u_{\eps}\|_{L^{2}}\lesssim\left(b_{\eps}(t)\right)^{-1}(\|u_{0}\|_{H^{s}}+\|u_{1}\|_{H^{s-1}})
\end{equation}
and
\begin{equation}\label{EQ: W06.02}
\|(\partial_{t}, \nabla_{x})u_{\eps}\|_{L^{2}}\lesssim\left(b_{\eps}(t)\right)^{-1+\frac{1}{2}}(\|u_{0}\|_{H^{s+1}}+\|u_{1}\|_{H^{s}})
\end{equation}
as $t\to\infty$ for arbitrary $s>0$.
\end{thm}

We will also need the following extension of Theorem \ref{TH: W06-1}, keeping its assumptions:

\begin{cor}\label{COR: W06-1}
The solution to \eqref{EQ: CPbb} satisfies the estimates
\begin{equation}\label{EQ: W06.03}
\|\partial_{t}^{l}\partial_{x}^{\alpha}u_{\eps}\|_{L^{2}}\lesssim\left(b_{\eps}(t)\right)^{-1+\frac{l}{2}}
(\|u_{0}\|_{H^{s+|\alpha|+l}}+\|u_{1}\|_{H^{s+|\alpha|+l-1}})
\end{equation}
as $t\to\infty$, for arbitrary $\alpha\in\mathbb N_{0}^{n}$ for all $s>0$ and for $l=0, 1$.
\end{cor}
\begin{proof}
Fix $s>0$ and $\alpha\in\mathbb N_{0}^{n}$. For $(u_{0}, u_{1})\in H^{s+|\alpha|}\times H^{s+|\alpha|-1}$, let us introduce new initial data by the formulae
$$
v_{0}:=\partial^{\alpha}u_{0}\in H^{s}, \,\,\,\,\, v_{1}:=\partial^{\alpha}u_{1}\in H^{s-1}.
$$
Then for these functions we consider the Cauchy problem
\begin{equation}\label{EQ-AP:02}
\left\{ \begin{split}
\partial_{t}^{2}v_{\eps}(t,x)-\Delta_{x}v_{\eps}(t,x)+\frac{b_{\eps}'(t)}{b_{\eps}(t)}\partial_{t}v_{\eps}(t,x)&=0, \; (t,x)\in [0, \infty)\times \mathbb R^{n},\\
v_{\eps}(0,x)&=v_{0}(x), \; x\in \mathbb R^{n}, \\
\partial_{t}v_{\eps}(0,x)&=v_{1}(x), \; x\in \mathbb R^{n}.
\end{split}
\right.
\end{equation}
Therefore, by Theorem \ref{TH: W06-1} the solution $v_{\eps}$ of the equation \eqref{EQ-AP:02} satisfies
\begin{equation*}\label{EQ: W06.01-2}
\|v_{\eps}\|_{L^{2}}\lesssim\left(b_{\eps}(t)\right)^{-1}(\|v_{0}\|_{H^{s}}+\|v_{1}\|_{H^{s-1}})
\end{equation*}
and
\begin{equation*}\label{EQ: W06.02-2}
\|(\partial_{t}, \nabla_{x})v_{\eps}\|_{L^{2}}\lesssim\left(b_{\eps}(t)\right)^{-1+\frac{1}{2}}(\|v_{0}\|_{H^{s+1}}+\|v_{1}\|_{H^{s}}).
\end{equation*}
Thus, we obtain
\begin{equation*}\label{EQ: W06.01-3}
\|\partial^{\alpha}u_{\eps}\|_{L^{2}}\lesssim\left(b_{\eps}(t)\right)^{-1}(\|\partial^{\alpha}u_{0}\|_{H^{s}}+\|\partial^{\alpha}u_{1}\|_{H^{s-1}})
\end{equation*}
and
\begin{equation*}\label{EQ: W06.02-3}
\|(\partial_{t}, \nabla_{x})\partial^{\alpha}u_{\eps}\|_{L^{2}}\lesssim\left(b_{\eps}(t)\right)^{-1+\frac{1}{2}}
(\|\partial^{\alpha}u_{0}\|_{H^{s+1}}+\|\partial^{\alpha}u_{1}\|_{H^{s}})
\end{equation*}
since the coefficients of the equation \eqref{EQ-AP:01} do not depend on $x$.
This proves Corollary \ref{COR: W06-1}.
\end{proof}

We now summarise the conclusions regarding the $L^{2}$-norms of the family $b_{\eps}$:

\begin{prop} \label{Lemma: Boundedness in eps}
Let $b$ be a positive distribution and let $\eps>0$.
Consider the following two cases:
\begin{itemize}
\item[(a)] $tb'_{\eps}(t)/b_{\eps}(t)\rightarrow\infty$ as $t\rightarrow\infty$;
\item[(b)] $\limsup_{t\to\infty} |tb'_{\eps}(t)/b_{\eps}(t)|<\infty$.
\end{itemize}
Then, respectively, for all $s>0$ we have
\begin{itemize}
\item[(a)] the solution of the Cauchy problem \eqref{EQ: CPbb} satisfies the estimate
\begin{equation*}\label{EQ: W07.01-2}
\|u_{\eps}\|_{L^{2}}\leq C(\|u_{0}\|_{H^{s}}+\|u_{1}\|_{H^{s-1}})
\end{equation*}
for some constant $C$ which is not depending on $b,\eps$ and initial data;
\item[(b)] the solution of the Cauchy problem \eqref{EQ: CPbb} has the decay
\begin{equation*}\label{EQ: W06.01-4}
\|u_{\eps}\|_{L^{2}}\leq C\left(b_{\eps}(t)\right)^{-1}(\|u_{0}\|_{H^{s}}+\|u_{1}\|_{H^{s-1}})
\end{equation*}
as $t\rightarrow\infty$, for some constant $C$ which is not depending on $b,\eps$ and initial data.
\end{itemize}
If $b$ is a positive, piecewise continuous and increasing function, then
the solution of the Cauchy problem \eqref{EQ: CPbb} is uniformly bounded in $\eps\in(0,1]$, i.e. for the solution $u_{\eps}$ the following estimate is true
\begin{equation}\label{EQ:ub}
\|u_{\eps}\|_{L^{2}}\leq C_{b, u_{0}, u_{1}},
\end{equation}
where the constant $C_{b, u_{0}, u_{1}}$ depends only on $b$ and initial data $u_{0}, u_{1}$.
\end{prop}

\begin{proof} ${\rm (a)}$ When $|\alpha|=0$ the inequality \eqref{EQ: W07.01} of Theorem \ref{TH: W07-1} implies the case ${\rm (a)}$.

${\rm (b)}$ This case follows from the inequality \eqref{EQ: W06.01} of Theorem \ref{TH: W06-1}.

If $b$ satisfies \eqref{EQ:as}, then we have \eqref{EQ:ub} using
cases ${\rm (a)}$ and ${\rm (b)}$ by taking into account the fact that $b_{\eps}(t)$ is also a positive piecewise continuous increasing function.
\end{proof}

\begin{lemma}
\label{LM: est-s in NonHomogeneous case}
Let us consider the nonhomogeneous equation
\begin{equation}\label{EQ-NonHomogeneous}
\left\{ \begin{split}
\partial_{t}^{2}u_{\eps}(t,x)-\Delta_{x} u_{\eps}(t,x)+\frac{b'_{\eps}(t)}{b_{\eps}(t)}\partial_{t}u_{\eps}(t,x)&=f_{\eps}(t,x), \; (t,x)\in [0, T]\times \mathbb R^{n},\\
u_{\eps}(0,x)&=u_{0}(x), \; x\in \mathbb R^{n}, \\
\partial_{t}u_{\eps}(0,x)&=u_{1}(x), \; x\in \mathbb R^{n}.
\end{split}
\right.
\end{equation}
Then, for any $t\in [0,T]$ and $s\in\mathbb R$ we have
\begin{equation}\label{EQ-NonHomogeneous-01}
\|\nabla u_{\eps}(t, \cdot)\|_{H^{s}}^2+\|\partial_{t}u_{\eps}(t, \cdot)\|_{H^{s}}^2\leq C_{1}(\|\nabla u_{0}\|^2_{H^{s}}+\|u_{1}\|_{H^{s}}^2)+C_{2}\int_{0}^{t}\|f_{\eps}(\tau, \cdot)\|_{H^{s}}^2d\tau.
\end{equation}
\end{lemma}
\begin{proof} We begin by applying the Fourier transform with respect to the spatial variables. This yields the Cauchy problem for the second ordinary differential equation
\begin{equation}\label{EQ-NonHomogeneous-03}
\left\{ \begin{split}
\partial_{t}^{2}\widehat{u_{\eps}}(t)+|\xi|^{2} \widehat{u}_{\eps}(t)+\frac{b'_{\eps}(t)}{b_{\eps}(t)}\partial_{t}\widehat{u}_{\eps}(t)&=\widehat{f}_{\eps}(t),\\
\widehat{u}_{\eps}(0)&=\widehat{u}_{0}, \\
\partial_{t}\widehat{u}_{\eps}(0)&=\widehat{u}_{1}.
\end{split}
\right.
\end{equation}

Now, instead of \eqref{EQ-NonHomogeneous-03} consider the following system
\begin{equation}\label{EQ-NonHomogeneous-04}
\partial_{t}U_{\eps}(t)=K_{\eps}(t)U_{\eps}(t)+F_{\eps}(t),
\end{equation}
with the Cauchy data
\begin{equation}\label{EQ-NonHomogeneous-04-Cauchy}
U_{\eps}(0)=U_{0},
\end{equation}
where
$$
U_{\eps}(t)=\left(
    \begin{array}{cc}
      i |\xi| \widehat{u_{\eps}}(t) \\
      \partial_{t}\widehat{u_{\eps}}(t) \\
           \end{array}
  \right), \,\,\,\,\,\,
U_{0} =\left(
    \begin{array}{cc}
      i |\xi| \widehat{u}_{0} \\
      \widehat{u}_{1} \\
           \end{array}
  \right),
$$
\begin{equation}\label{EQ-NonHomogeneous-04-Matrix}
K_{\eps}(t)=\left(
    \begin{array}{cc}
      0 & i |\xi|\\
      i |\xi| & - b'_{\eps}(t)/b_{\eps}(t) \\
           \end{array}
  \right),
\end{equation}
and
$$
F_{\eps}(t) =\left(
    \begin{array}{cc}
      0 \\
      \widehat{f}_{\eps}(t) \\
           \end{array}
  \right).
$$
The matrix \eqref{EQ-NonHomogeneous-04-Matrix} is symmetric. This allows us to use statements of e.g. Taylor's book \cite[Chapter IV]{T81} (also see \cite[Case 1]{GR15b} and \cite{GR12}). Hence, we obtain the statement of the lemma.
\end{proof}


Now let us formulate the main results of this paper. As described in the introduction it will be natural to make the assumptions \eqref{EQ:as} for our analysis, ensuring that $b'/b$ is positive.

\begin{thm}[Existence]
\label{TH: VWS-01}
Assume that the coefficient $b$ of the Cauchy problem \eqref{EQ-AP:01} is a positive, piecewise continuous and increasing function
such that $b\ge b_{0}$ for some constant $b_{0}>0$, and that $b'$ is a positive distribution such that $b'\ge b_{0}'$ for some constant $b_{0}'>0$.
Let $s>0$ and let the Cauchy data $(u_0, u_1)$ be in ${H}^{s}\times {H}^{s-1}$.
Then the Cauchy problem \eqref{EQ-AP:01} has a very weak solution of order $s$.
\end{thm}
The uniqueness of very weak solutions will be formulated in Theorem \ref{TH: consistency-01}.

Now we formulate the theorem saying that very weak solutions recapture the classical solutions in the case the latter exist. This happens, for example, under conditions of Theorems \ref{TH: W07-1} and \ref{TH: W06-1}. So, we can compare the solution given by Theorems \ref{TH: W07-1} and \ref{TH: W06-1} with the very weak solution in Theorem \ref{TH: VWS-01} under assumptions when Theorems \ref{TH: W07-1} and \ref{TH: W06-1} hold.

\begin{thm}[Consistency]
\label{TH: consistency}
Assume that $b\in C^{1}([0, \infty))$ is an increasing function such that $b\ge b_0>0$, and that $b'$ is a positive function such that $b'\ge b_{0}'$ for some constant $b_{0}'>0$.
Let $s>0$, and consider the Cauchy problem
\begin{equation}\label{EQ: Consistency-01}
\left\{ \begin{split}
\partial_{t}^{2}u(t,x)-\Delta_{x}u(t,x)+\frac{b'(t)}{b(t)}\partial_{t}u(t,x)&=0, \; (t,x)\in [0, \infty)\times \mathbb R^{n},\\
u(0,x)&=u_{0}(x), \; x\in \mathbb R^{n}, \\
\partial_{t}u(0,x)&=u_{1}(x), \; x\in \mathbb R^{n},
\end{split}
\right.
\end{equation}
with $(u_0,u_1)\in {H}^{s}\times {H}^{s-1}$. Let $u$ be a very weak solution of
\eqref{EQ: Consistency-01}. Then for any regularising family $b_{\eps}$ in Definition \ref{def_vws}, the representatives $(u_\eps)_\eps$ of $u$ converge in the space $C([0, T]; {H}^{s})\cap
C^1([0, T]; {H}^{s-1})$ as $\eps\rightarrow0$
to the unique classical solution in $C([0, T]; {H}^{s}) \cap
C^1([0, T]; {H}^{s-1})$ of the Cauchy problem \eqref{EQ: Consistency-01}
given by Theorems \ref{TH: W07-1} and \ref{TH: W06-1}, for any $T>0$.
\end{thm}

We note that the convergence in Theorem \ref{TH: consistency} can be realised also on the interval $[0,\infty)$ depending on further properties of $b(t)$ allowing a global application of Lemma
\ref{LM: est-s in NonHomogeneous case}, see Remark \ref{REM:global}.

\section{Proof of Theorem \ref{TH: VWS-01}}

We regularise $b$ by the convolution with a mollifier in $C^\infty_0(\mathbb R)$ and get nets of smooth functions as coefficients. More precisely, let $\psi\in C^\infty_0(\mathbb R)$, $\psi\ge 0$ with $\int\psi=1$.

Define
$$
\psi_{\eps}(t):=\frac{1}{\eps}\psi\left(\frac{t}{\eps}\right),
$$
and
$$
b_{\eps}(t):=(b\ast \psi_{\eps})(t), \qquad b_{\eps}'(t):=(b'\ast \psi_{\eps})(t),  \qquad t\geq 0.
$$
Since $b$ and $b'$ are positive distributions
and $\psi\in C^\infty_0(\mathbb R)$, $\textrm{supp}\,\psi\subset\textsc{K}$, $\psi\ge 0$,
then we have
\begin{align*}
b_{\eps}(t)&=(b\ast \psi_{\eps})(t)=\int\limits_{\mathbb R}b(t-\tau)\psi_{\eps}(\tau)d\tau=\int\limits_{\mathbb R}b(t-\eps\tau)\psi(\tau)d\tau \\
&=\int\limits_{\textsc{K}}b(t-\eps\tau)\psi(\tau)d\tau\geq b_{0} \int\limits_{\textsc{K}}\psi(\tau)d\tau=b_{0}>0,
\end{align*}
and
\begin{align*}
b_{\eps}'(t)&=(b'\ast \psi_{\eps})(t)=\int\limits_{\mathbb R}b'(t-\tau)\psi_{\eps}(\tau)d\tau=\int\limits_{\mathbb R}b'(t-\eps\tau)\psi(\tau)d\tau \\
&=\int\limits_{\textsc{K}}b'(t-\eps\tau)\psi(\tau)d\tau\geq b_{0}' \int\limits_{\textsc{K}}\psi(\tau)d\tau=b_{0}'>0.
\end{align*}

By the structure theorem for compactly supported distributions, we have that there exist $L\in\mathbb N_{0}$ and $c>0$ such that
\begin{equation}\label{EQ: a-q-C-moderate}
|\partial^{k}_{t}b_{\eps}(t)|\le c\,\eps^{-L-k},
\end{equation}
for all $k\in\mathbb N_{0}$ and $t\in[0,T]$. We note that the numbers $L$ may be related to the distributional orders of $b$.

Now, let us define $L$ more precisely under the assumption \eqref{EQ:as}. Indeed, when
$b$ is a piecewise continuous, positive and increasing function, we obtain
\begin{multline*}
|b_{\eps}(t)|=|(b\ast \psi_{\eps})(t)|=\left|\int\limits_{\mathbb R}b(t-\tau) \psi_{\eps}(\tau)d\tau\right|=\left|\int\limits_{\mathbb R}b(t-\eps\tau)\psi(\tau)d\tau\right| \\
=\left|\int\limits_{\textsc{K}}b(t-\eps\tau)\psi(\tau)d\tau\right|\leq c \int\limits_{\textsc{K}}\psi(\tau)d\tau=c,
\end{multline*}
and
\begin{align*}
|b_{\eps}'(t)|=|(b'\ast &\psi_{\eps})(t)|=\left|\int\limits_{\mathbb R}b'(t-\tau)\psi_{\eps}(\tau)d\tau\right|=\frac{1}{\eps}\left|\int\limits_{\mathbb R}b(t-\tau)\psi_{\eps}'(\tau)d\tau\right|\\
&=\frac{1}{\eps}\left|\int\limits_{\mathbb R}b(t-\eps\tau)\psi'(\tau)d\tau\right| =\frac{1}{\eps}\left|\int\limits_{\textsc{K}}b(t-\eps\tau)\psi'(\tau)d\tau\right|\leq c_1\eps^{-1},
\end{align*}
for all $t\in[0,T]$, where $c, c_1$ depend only on $T$.

Thus, for $b$, which is a piecewise continuous, positive and increasing function, $L=0$, and for the distributional function ($\delta$--like) $b'$, we have $L=1$.

Hence, $b_{\eps}, b_{\eps}'$ are $C^\infty$--moderate regularisations of the coefficients $b, b'$.
Now, fix $\eps\in(0,1]$, and consider the regularised problem
\begin{equation}\label{EQ: CPbb-2}
\left\{ \begin{split}
\partial_{t}^{2}u_{\eps}(t,x)-\Delta_{x}u_{\eps}(t,x)+\frac{b'_{\eps}(t)}{b_{\eps}(t)}\partial_{t}u_{\eps}(t,x)&=0, \; (t,x)\in [0, \infty)\times \mathbb R^{n},\\
u_{\eps}(0,x)&=u_{0}(x), \; x\in \mathbb R^{n}, \\
\partial_{t}u_{\eps}(0,x)&=u_{1}(x), \; x\in \mathbb R^{n},
\end{split}
\right.
\end{equation}
with the Cauchy data satisfy
$(u_0,u_1)\in {H}^{s}\times {H}^{s-1}$ and $b_{\eps}, b_{\eps}'\in C^\infty([0, \infty))$.
Then by Theorem \ref{TH: W07-1} the equation \eqref{EQ: CPbb-2} has a unique solution in the space $C^{0}([0, \infty);{H}^{s})\cap C^{1}([0, \infty);{H}^{s-1})$.
In fact, this unique solution is from $C^{k}([0, \infty);{H}^{s-k})$ for all $k$.

Now the proof is depending on the behavior of $b'_{\eps}(t)/b_{\eps}(t)$ at $t\rightarrow\infty$.
Let us consider several cases: when $t b'_{\eps}(t)/b_{\eps}(t)\rightarrow\infty$
as $t\rightarrow\infty$, and when $t b'_{\eps}(t)/b_{\eps}(t)=O(1)$
or $t b'_{\eps}(t)/b_{\eps}(t)=o(1)$ as $t\rightarrow\infty$.
Anyway applying Theorem \ref{TH: W07-1}, or Theorem \ref{TH: W06-1}, or Corollary \ref{COR: W06-1} to the equation \eqref{EQ: CPbb-2},
using the inequality \eqref{EQ: a-q-C-moderate} and that
$$
\frac{1}{|b_{\eps}(t)|}\leq\frac{1}{\tilde{b}_{0}},
$$
we get the estimate
\begin{equation}\label{ES: exp-1}
\| \partial_t u_{\eps}(t,\cdot)\|_{{H}^{|\alpha|}}^2\leq
C (\| u_0\|_{{H}^{s+|\alpha|+1}}^2+\|u_1\|_{{H}^{s+|\alpha|}}^2)
\end{equation}
uniformly in $\eps\in(0, 1]$.


The fact that there exist $N\in\mathbb N_{0}$, $c>0$ and, for all $k\in\mathbb N_{0}$ there exist
$c_k>0$ such that
$$
\|\partial_t^k u_\eps(t,\cdot)\|_{{H}^{s-k}}\le c_k \eps^{-N-k},
$$
for all $t\in[0, \infty)$, and $\eps\in(0,1]$, follows from the estimate
$$
\left|\partial_t^k\left(\frac{b'_{\eps}(t)}{b_{\eps}(t)}\right)\right|\le c_k \eps^{-1-k}
$$
which holds for all $t\in[0, \infty)$, $k\in\mathbb N_{0}$, and $\eps\in(0,1]$, from Theorems \ref{TH: W07-1} and \ref{TH: W06-1}, Corollary \ref{COR: W06-1}, and acting by the iterations of $\partial_{t}$ and by $\Delta_{x}$ on the
equality
$$
\partial_{t}^{2}u_{\eps}(t,x)+\frac{b'_{\eps}(t)}{b_{\eps}(t)}\partial_{t}u_{\eps}(t,x)=\Delta_{x}u_{\eps}(t,x),
$$
and taking it in $L^{2}$--norms. It means that $u_{\eps}$ is from the space
$C^{k}([0, \infty);{H}^{s-k})$ for any $k\in\mathbb N_{0}$, i.e. is
$C^\infty([0, \infty);{H}^{s})$-moderate.

This shows that the Cauchy problem \eqref{EQ-AP:01} has a very weak solution.

\section{Proof of Theorem \ref{TH: consistency}}
\label{SEC:pf1}

Here we prove the consistency of the very weak solution with the
classical one when the coefficients are
regular enough. But first, we show that the very weak solution is unique in an appropriate sense. To present uniqueness we will use the notions of Colombeau algebras adapted to the properties of classical solutions.

\begin{defi}
\label{DEF: negligible} We say that $(u_\eps)_\eps$ is
\emph{$C^\infty$-negligible} if for all $K\Subset\mathbb R$, for
all $\alpha\in\mathbb N$ and for all $\ell\in\mathbb N$ there exists
a constant $c>0$ such that
$$
\sup_{t\in K}|\partial^\alpha u_\eps(t)|\le c\eps^{\ell},
$$
for all $\eps\in(0,1]$.

Also, we call $(f_{\eps})_{\eps}$ is \emph{$C^{\infty}(\mathbb R_{+}; H^s)$--negligible} if for any $K\Subset\mathbb R$, for
all $k\in\mathbb N$ and for all $p\in\mathbb N$ there exists
a constant $c_{1}>0$ such that
$$
\sup_{t\in K}\|\partial^{k} f_\eps(t,\cdot)\|_{H^{s-k}}\le c_{1}\eps^{p},
$$
for arbitrary $\eps\in(0,1]$ and for any $k\in\mathbb N_0$.
\end{defi}

We now introduce the Colombeau algebra as the quotient
$$
\mathcal G(\mathbb R)=\frac{C^\infty-\text{moderate\,
nets}}{C^\infty-\text{negligible\, nets}}.
$$
For the general analysis of $\mathcal G(\mathbb R)$ we refer to
e.g. Oberguggenberger \cite{Oberguggenberger:Bk-1992}.


\begin{thm}[Uniqueness]
\label{TH: consistency-01}
Assume that $b$ is a piecewise continuous increasing function such that $b\ge b_{0}$ for some constant $b_{0}>0$, and that $b'$ is a positive distribution such that $b'\ge b_{0}'$ for some constant $b_{0}'>0$. Let $(u_0,u_1)\in {H}^{s} \times {H}^{s-1}$ for some $s\in\mathbb R_{+}$.
Then there exists an embedding of the coefficients $b, b'$ into $\mathcal G(\mathbb R_{+})$,
such that the Cauchy problem \eqref{EQ-AP:01}, that is
\begin{equation*}
\left\{ \begin{split}
\partial_{t}^{2}u(t,x)-\Delta_{x}u(t,x)+\frac{b'(t)}{b(t)}\partial_{t}u(t,x)&=0, \; (t,x)\in \mathbb R_{+}\times \mathbb R^{n},\\
u(0,x)&=u_{0}(x), \; x\in \mathbb R^{n}, \\
\partial_{t}u(0,x)&=u_{1}(x), \; x\in \mathbb R^{n},
\end{split}
\right.
\end{equation*}
has a unique solution $u\in
\mathcal G(\mathbb R_{+}; H^s)$.
\end{thm}

Here $\mathcal G(\mathbb R_{+}; H^s)$ stands for the space of families which are
in $\mathcal G(\mathbb R_{+})$ with respect to $t$ and in $H^s$ with respect to $x$.

\begin{proof}
Let us show that by embedding the coefficient in the
corresponding Colombeau algebras the Cauchy problem has a unique
solution $u\in\mathcal G(\mathbb R_{+}; H^s)$.
The existence follows along the lines of the proof of Theorem \ref{TH: VWS-01}.
Assume now that the Cauchy problem has another
solution $v\in\mathcal G(\mathbb R_{+}; H^s)$. At the level of
representatives this means
\begin{equation*}
\left\{ \begin{split}
\partial_{t}^{2}(u_\eps-v_\eps)(t,x)-\Delta_{x}(u_\eps-v_\eps)(t,x)+\frac{b'_{\eps}(t)}{b_{\eps}(t)}\partial_{t}(u_\eps-v_\eps)(t,x)&=f_{\eps}(t,x), \\
(u_\eps-v_\eps)(0,x)&=0,  \\
(\partial_{t}u_\eps-\partial_{t}v_\eps)(0,x)&=0,
\end{split}
\right.
\end{equation*}
with
$$
f_{\eps}(t,x)=\left(\frac{\widetilde{b'}_{\eps}(t)}{\widetilde{b}_{\eps}(t)}-\frac{b'_{\eps}(t)}{b_{\eps}(t)}\right) \partial_{t}v_\eps(t,x),
$$
where $(\widetilde{b'}_\eps)_{\eps}$ and $(\widetilde{b}_\eps)_{\eps}$ are approximations corresponding to $v_\eps$. Indeed, $f_{\eps}$ is $C^{\infty}(\mathbb R_{+}; H^s)$--negligible since
\begin{align*}
\frac{\widetilde{b'}_{\eps}(t)}{\widetilde{b}_{\eps}(t)}-\frac{b'_{\eps}(t)}{b_{\eps}(t)}&=
\frac{\widetilde{b'}_{\eps}(t)b_{\eps}(t)-b'_{\eps}(t)\widetilde{b}_{\eps}(t)}{\widetilde{b}_{\eps}(t)b_{\eps}(t)}=\\
=&\frac{\widetilde{b'}_{\eps}(t)b_{\eps}(t)-b'_{\eps}(t)b_{\eps}(t)+b'_{\eps}(t)b_{\eps}(t)-b'_{\eps}(t)\widetilde{b}_{\eps}(t)}{\widetilde{b}_{\eps}(t)b_{\eps}(t)}=\\
=&\frac{(\widetilde{b'}_{\eps}(t)-b'_{\eps}(t))b_{\eps}(t)+(b_{\eps}(t)-\widetilde{b}_{\eps}(t))b'_{\eps}(t)}{\widetilde{b}_{\eps}(t)b_{\eps}(t)},
\end{align*}
and, the differences
$(\widetilde{b'}_{\eps}(t)-b'_{\eps}(t))_{\eps}$ and $(b_{\eps}(t)-\widetilde{b}_{\eps}(t))_{\eps}$ are $C^{\infty}(\mathbb R_{+})$--negligible being from the same equivalence class in the Colombeau algebra.

Hence, using Lemma \ref{LM: est-s in NonHomogeneous case} globally in $T$, we get the statement of the theorem
since the function $(f_{\eps})_{\eps}$ is $C^{\infty}(\mathbb R_{+}; H^s)$--negligible.
\end{proof}

\begin{proof}[Proof of Theorem \ref{TH: consistency}] Now we compare the classical solution $\widetilde{u}$ provided  by Theorems \ref{TH: W07-1} and \ref{TH: W06-1} with the very weak solution $u$ given by Theorem \ref{TH: consistency}.
The classical solution satisfies
\begin{equation}\label{EQ:Consistency-02}
\left\{ \begin{split}
\partial_{t}^{2}\widetilde{u}(t,x)-\Delta_{x}\widetilde{u}(t,x)+\frac{b'(t)}{b(t)}\partial_{t}\widetilde{u}(t,x)&=0, \; (t,x)\in \mathbb R_{+}\times \mathbb R^{n},\\
\widetilde{u}(0,x)&=u_{0}(x), \; x\in \mathbb R^{n}, \\
\partial_{t}\widetilde{u}(0,x)&=u_{1}(x), \; x\in \mathbb R^{n}.
\end{split}
\right.
\end{equation}
For the very weak solution $u$, there is a representative $(u_\eps)_\eps$ of $u$ such
that
\begin{equation}\label{EQ:Consistency-03}
\left\{ \begin{split}
\partial_{t}^{2}u_{\eps}(t,x)-\Delta_{x}u_{\eps}(t,x)+\frac{b'_{\eps}(t)}{b_{\eps}(t)}\partial_{t}u_{\eps}(t,x)&=0, \; (t,x)\in \mathbb R_{+}\times \mathbb R^{n},\\
u_{\eps}(0,x)&=u_{0}(x), \; x\in \mathbb R^{n}, \\
\partial_{t}u_{\eps}(0,x)&=u_{1}(x), \; x\in \mathbb R^{n},
\end{split}
\right.
\end{equation}
for a suitable embedding of the coefficient $b$. Indeed, for $b\in C^{1}([0,T])$ the
nets $(\frac{b'_{\eps}}{b_{\eps}}-\frac{b'}{b})_\eps$ are converging to $0$ in
$C([0,T])$. Let us rewrite
\eqref{EQ:Consistency-02} as
\begin{equation}\label{EQ:Consistency-04}
\left\{ \begin{split}
\partial_{t}^{2}\widetilde{u}(t,x)-\Delta_{x}\widetilde{u}(t,x)+\frac{b'_{\eps}(t)}{b_{\eps}(t)}\partial_{t}\widetilde{u}(t,x)&=n_{\eps}(t,x), \; (t,x)\in \mathbb R_{+}\times \mathbb R^{n},\\
\widetilde{u}(0,x)&=u_{0}(x), \; x\in \mathbb R^{n}, \\
\partial_{t}\widetilde{u}(0,x)&=u_{1}(x), \; x\in \mathbb R^{n},
\end{split}
\right.
\end{equation}
where
$$
n_{\eps}(t,x)=\left(\frac{b'_{\eps}(t)}{b_{\eps}(t)}-\frac{b'(t)}{b(t)}\right)\partial_{t}\widetilde{u}(t,x),
$$
and $n_\eps\in C([0,T]; H^{s-1})$. Also $n_\eps\rightarrow0$ as $\eps\to0$ in $C([0,T]; H^{s-1})$. From \eqref{EQ:Consistency-03} and \eqref{EQ:Consistency-04} we get that $(\widetilde{u}-u_\eps)$ solves the Cauchy problem

\begin{equation*}
\left\{ \begin{split}
\partial_{t}^{2}(\widetilde{u}-u_\eps)(t,x)-\Delta_{x}(\widetilde{u}-u_\eps)(t,x)+\frac{b'_{\eps}(t)}{b_{\eps}(t)}\partial_{t}(\widetilde{u}-u_\eps)(t,x)&=n_{\eps}(t,x),\\
(\widetilde{u}-u_\eps)(0,x)&=0, \\
(\partial_{t}\widetilde{u}-\partial_{t}u_\eps)(0,x)&=0.
\end{split}
\right.
\end{equation*}
Since by Lemma \ref{LM: est-s in NonHomogeneous case} we have
$$
\|(\widetilde{u}-u_\eps)\|_{H^{s}}\lesssim\|n_{\eps}\|_{C([0,T]; H^{s-1})},
$$
it follows that $u_\eps\to \widetilde{u}$
in $C([0,T]; {H}^{s}) \cap
C^1([0, T]; {H}^{s-1})$.
\end{proof}

\begin{rem}\label{REM:global}
The convergence can be made on the interval $[0,\infty)$ compared to the statement of Theorem \ref{TH: consistency}, depending on further properties of $b(t)$.
We note that in the proof of Theorem \ref{TH: consistency} we apply Lemma
\ref{LM: est-s in NonHomogeneous case} with
$$
f_{\eps}(t,x)=\left(\frac{\widetilde{b'}_{\eps}(t)}{\widetilde{b}_{\eps}(t)}-\frac{b'_{\eps}(t)}{b_{\eps}(t)}\right) \partial_{t}v_\eps(t,x).
$$
To ensure the convergence of the integral in \eqref{EQ-NonHomogeneous-01}, for example under conditions of Theorem \ref{TH: W06-1}, using
the estimate
$$\|\partial_{t}v_{\eps}\|_{L^{2}}\lesssim\left(b_{\eps}(t)\right)^{-1+\frac{1}{2}}(\|u_{0}\|_{H^{s+1}}+\|u_{1}\|_{H^{s}}),$$ we get that
$$
\int_0^\infty \|f_{\eps}(\tau, \cdot)\|_{L^{2}}^2d\tau \lesssim
\int_0^\infty \frac{|b'_{\eps}(\tau)|^2}{|b_{\eps}(\tau)|^3} d\tau\, (\|u_{0}\|_{H^{s+1}}^2+\|u_{1}\|_{H^{s}}^2)<\infty
$$
is finite under the conditions of Theorem \ref{TH: W06-1}.
Similarly, under conditions of Theorem \ref{TH: W07-1}, using that
$$\|\partial_{t}v_{\eps}\|_{L^{2}}\lesssim\frac{b_{\eps}(t)}{b'_{\eps}(t)}\left(1+\int\limits_{0}^{t} \frac{b_{\eps}(\tau)}{b'_{\eps}(\tau)}d\tau\right)^{-1}(\|u_{0}\|_{H^{s+1}}+\|u_{1}\|_{H^{s}}),$$
we get that
$$
\int_0^\infty \|f_{\eps}(\tau, \cdot)\|_{L^{2}}^2d\tau \lesssim
\int_0^\infty \left(1+\int\limits_{0}^{\tau} \frac{b_{\eps}(\sigma)}{b'_{\eps}(\sigma)}d\sigma\right)^{-2} d\tau \,(\|u_{0}\|_{H^{s+1}}^2+\|u_{1}\|_{H^{s}}^2)
$$
is finite depending on further properties of $b(t)$.
\end{rem}

\section{Distributional initial data case}
\label{SEC:distr}

In this section we briefly discuss the case when the Cauchy data are distributional. More precisely, we consider the following initial problem
\begin{equation}\label{EQ-AP:01-D}
\left\{ \begin{split}
\partial_{t}^{2}u(t,x)-\Delta_{x}u(t,x)+\frac{b'(t)}{b(t)}\partial_{t}u(t,x)&=0, \; (t,x)\in [0, \infty)\times \mathbb R^{n},\\
u(0,x)&=u_{0}(x), \; x\in \mathbb R^{n}, \\
\partial_{t}u(0,x)&=u_{1}(x), \; x\in \mathbb R^{n},
\end{split}
\right.
\end{equation}
as in \eqref{EQ-AP:01}, but here the Cauchy data $(u_{0}, u_{1})$ are allowed to be from ${H}^{s}\times{H}^{s-1}$ with a negative $s$. Indeed, the main argument in allowing $s$ to be also negative is the fact that the results of Theorems \ref{TH: W07-1} and \ref{TH: W06-1} could be extended as follows:

\vspace{3mm}

Assume that $s:=-\sigma$, where $\sigma>0$, i. e. $(u_{0}, u_{1})\in{H}^{-\sigma+1}\times{H}^{-\sigma}$. Then
\begin{itemize}
\item Firstly, we put
$$v:=\langle D_{x}\rangle^{-\sigma} u, \,\,\, v_{0}:=\langle D_{x}\rangle^{-\sigma} u_{0}, \,\,\, v_{1}:=\langle D_{x}\rangle^{-\sigma} u_{1},$$
where $\langle D_{x}\rangle^{-\sigma}$ is a Fourier multiplier with the symbol $(1+|\xi|^{2})^{-\frac{\sigma}{2}}$, ($\sigma>0$);
\item Secondly, we consider the Cauchy problem \eqref{EQ-AP:01-D} with $v$, that is
\begin{equation}\label{EQ-AP:01-D-2}
\left\{ \begin{split}
\partial_{t}^{2}v(t,x)-\Delta_{x}v(t,x)+\frac{b'(t)}{b(t)}\partial_{t}v(t,x)&=0, \; (t,x)\in [0, \infty)\times \mathbb R^{n},\\
v(0,x)&=v_{0}(x), \; x\in \mathbb R^{n}, \\
\partial_{t}v(0,x)&=v_{1}(x), \; x\in \mathbb R^{n}.
\end{split}
\right.
\end{equation}
\item Finally, for $\sigma>0$ we recall the Sobolev spaces
$${H}^{-\sigma}:=\{f: \,\, \langle D\rangle^{-\sigma} f\in L^{2}\}.$$
\end{itemize}
Since the coefficients of the equation \eqref{EQ-AP:01-D-2} do not depend on $x$ the statements of Theorems \ref{TH: W07-1} and \ref{TH: W06-1} hold for any $s$. Thus, arguing in the same way as in the previous sections, we get the following results on very weak solutions:

\begin{thm}[Existence]
\label{TH: VWS-01-D}
Assume that the coefficient $b$ of the Cauchy problem \eqref{EQ-AP:01-D} is a positive, piecewise continuous and increasing function
such that $b\ge b_{0}$ for some constant $b_{0}>0$, and that $b'$ is a positive distribution such that $b'\ge b_{0}'$ for some constant $b_{0}'>0$.
Let $s\in\mathbb R$ and let the Cauchy data $(u_0, u_1)$ be in ${H}^{s}\times {H}^{s-1}$.
Then the Cauchy problem \eqref{EQ-AP:01-D} has a very weak solution of order $s$.
\end{thm}

\begin{thm}[Consistency]
\label{TH: consistency-D}
Assume that $b\in C^{1}([0, \infty))$ is an increasing function such that $b\ge b_0>0$, and that $b'$ is a positive distribution such that $b'\ge b_{0}'$ for some constant $b_{0}'>0$.
Let $s\in\mathbb R$, and consider the Cauchy problem \eqref{EQ-AP:01-D} with $(u_0,u_1)\in {H}^{s}\times {H}^{s-1}$. Let $u$ be a very weak solution of
\eqref{EQ-AP:01-D}. Then for any regularising family $b_{\eps}$ in Definition \ref{def_vws}, the representatives $(u_\eps)_\eps$ of $u$ converge in $C([0, T],{H}^{s})\cap
C^1([0, T]; {H}^{s-1})$ as $\eps\rightarrow0$
to the unique classical (Sobolev) solution in $C([0, T],{H}^{s}) \cap
C^1([0, T],{H}^{s-1})$ of the Cauchy problem \eqref{EQ-AP:01-D}
given by Theorems \ref{TH: W07-1} and \ref{TH: W06-1}, for any $T>0$.
\end{thm}

\begin{thm}[Uniqueness]
\label{TH: consistency-01b}
Assume that $b$ is a piecewise continuous increasing function such that $b\ge b_{0}$ for some constant $b_{0}>0$, and that $b'$ is a positive distribution such that $b'\ge b_{0}'$ for some constant $b_{0}'>0$. Let $(u_0,u_1)\in {H}^{s} \times {H}^{s-1}$ for some $s\in\mathbb R$.
Then there exists an embedding of the coefficients $b, b'$ into $\mathcal G(\mathbb R_{+})$,
such that the Cauchy problem \eqref{EQ-AP:01-D} has a unique solution $u\in
\mathcal G(\mathbb R_{+}; H^s)$.
\end{thm}

We note that certain further generalisations of the obtained results are possible replacing the Laplacian with more general operators with a control on its spectral behaviour, using methods developed in \cite{RT16a, RT17}, or to dissipative wave equations on manifolds (\cite{DRT16}). Such questions will be addressed elsewhere.

\section{Numerical experiments}\label{Numerics}

To illustrate the theoretical results presented above, in this section we compute a sequence of solutions $(u_\epsilon)_{\epsilon>0}$ of the regularised problem \eqref{EQ: CPbb}, obtained for the case of a regularisation of the synthetic piecewise continuous increasing function
\[
b(t)=\begin{cases}
          & 1,  ~~~~~~~~~ t<5, \\
          & \frac{1}{10} t + \frac{3}{2}, ~~~t\geq 5.
         \end{cases}
\]
We consider the regularisation
$b_\epsilon(t)= (b \ast \psi_\epsilon)(t)$, $t\geq 0$, of the coefficient $b(t)$ by the convolution with the mollifier
\begin{equation}\label{mollifier1}
\psi_\epsilon(t) = \frac{1}{\epsilon} \psi \Big( \frac{t}{\epsilon} \Big),
\end{equation}
where
\[
\psi(t)= \begin{cases}
               &  \frac{1}{C} \exp \Big(  \frac{1}{t^2 -1}   \Big), ~~~|t|\leq 1, \\
               & 0, ~~~|t| >1.
              \end{cases}
\]
Here $C = 0.443994$ so that $\int \psi =1$.  We take the initial conditions for the regularised problem \eqref{EQ: CPbb} to be
\[
u_\epsilon(0,x)= \exp\Big( -\frac{(x-a)^2}{\delta} \Big), ~~~~~\partial_t u_\epsilon (0,x)= -\frac{1}{b_\epsilon (0)} \partial_x \Big( \exp \Big(-\frac{(x-a)^2}{\delta} \Big) \Big),
\]
where $\delta=0.3$ and $a=0$. In particular, in Figure \ref{Coefficientb}, we illustrate the regularisation $b_\epsilon(t)$ of the coefficient $b(t)$ for $\epsilon =0.5$.

\begin{figure}[ht!]
\centering
\includegraphics[width=13cm]{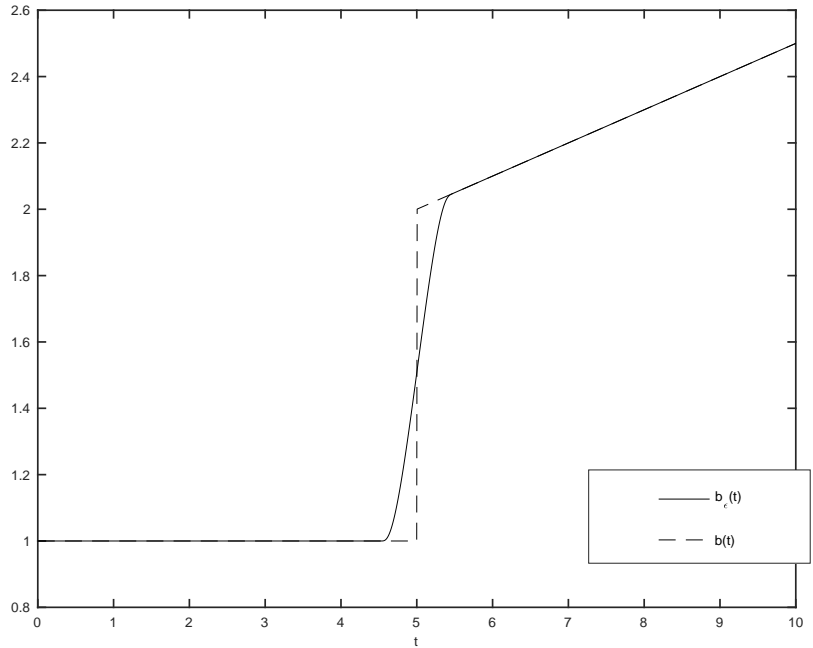}
\caption{In this plot, we display the function $b(t)$ together with its regularisation $b_\epsilon(t)$ obtained by convolution with the mollifier $\psi_\epsilon(t)$ defined in \eqref{mollifier1}. The regularised coefficient $b_\epsilon$ was computed numerically with Matlab R2016b. }
\label{Coefficientb}
\end{figure}
To approximate the solutions of the regularised problem \eqref{EQ: CPbb}, we solve the equivalent system
\begin{equation}\label{PUequation}
\left\{ \begin{split}
& p_t = -b_\epsilon (t) u_x, \\
& u_t = - \frac{1}{b_\epsilon(t)} p_x,
\end{split} \right.
\end{equation}
subject to the initial conditions
$$p(0,x)=u(0,x)=\exp\Big( -\frac{(x-a)^2}{\delta} \Big).$$
Observe that if $(p,u)$ is a solution of problem \eqref{PUequation}, then
$u_\epsilon=u$ is a solution of problem \eqref{EQ: CPbb}.

To solve numerically the initial value problem \eqref{PUequation}, we approximate the first derivatives $p_x, u_x$ with respect to the variable $x$ with a fourth-order finite difference scheme, and the classical fourth-order Runge-Kutta method is used for time stepping. The results for $u_\epsilon(t,x)$ at time $t=60$ are presented in Figure \ref{CompUeps}, for several values of the parameter $\epsilon$. The numerical parameters used in these computer simulations were $\Delta x=0.0171$, $\Delta t=0.0067$, and the spatial computational domain was the interval $[-50,70]$. In all computer simulations, we used Matlab R2016b.

In Figure \ref{propagationU2}, we display the solution $u_\epsilon(t,x)$ for $\epsilon=0.01$ at different values of time $t= 4.8,  5.0, 5.2, 5.4, 5.6, 5.8, 6.0, 6.2 , 6.4, 6.6 , 6.8, 7.0, 7.2$ before and after the interaction with the discontinuity at $t=5$. Observe that the profile starts splitting into two approximately at $t=5.6$, generating a second traveling wave with smaller size moving to the left.

In Figure \ref{AmplitudeUeps05} we illustrate the decay with respect to time $t$ of the solution $u_\epsilon$ of the regularised problem \eqref{EQ: CPbb} for some values of the parameter $\epsilon$.

Finally, in order to illustrate the evolution of a near-singular initial pulse in problem \eqref{EQ: CPbb}, in Figures \ref{propagationU3}, \ref{propagationU4}, \ref{propagationU5},
we display the solution $u_\epsilon$ of the regularised
problem \eqref{EQ: CPbb} at time $t=8$,  with $\epsilon=0.01$ and initial data in the form
\begin{equation}\label{Cauchydata}
u(0,x)=f(x), ~~~u_t(0,x)=-\frac{1}{b_\epsilon(0)} f'(x),
\end{equation}
where
\[
f(x)= \frac{1}{\pi} \Big( \frac{e}{x^2 + e^2} \Big),
\]
with $e=0.05$, $e=0.03$ and $e=0.01$, respectively. The coefficient $b(t)$ and its regularisation $b_\epsilon(t)$ were the same as in the previous experiments.  The numerical parameters were $\Delta t=0.0011$, $\Delta x=0.025$ for the experiment with $e=0.05$,  and $\Delta t=8E-4$, $\Delta x=0.025$, for $e=0.03$ and $\Delta t= 2.28E-4$, $\Delta x=0.008$ for the case of the experiment with $e=0.01$. In all experiments the spatial computational domain was the interval $[-20,20]$. Thus, we are considering a sequence of incoming pulses that approaches to the delta function. From these pictures, we can see that the left-going traveling wave (already observed in the experiments in Figure \ref{propagationU2}) does not appear to be regularised, and instead resembles a delta function but just smaller in amplitude. We point out that this phenomenon is different from conical refraction when the splitting singularity is more regular than the original one.

\begin{figure}[ht!]
\centering
\includegraphics[width=15cm,height=10cm]{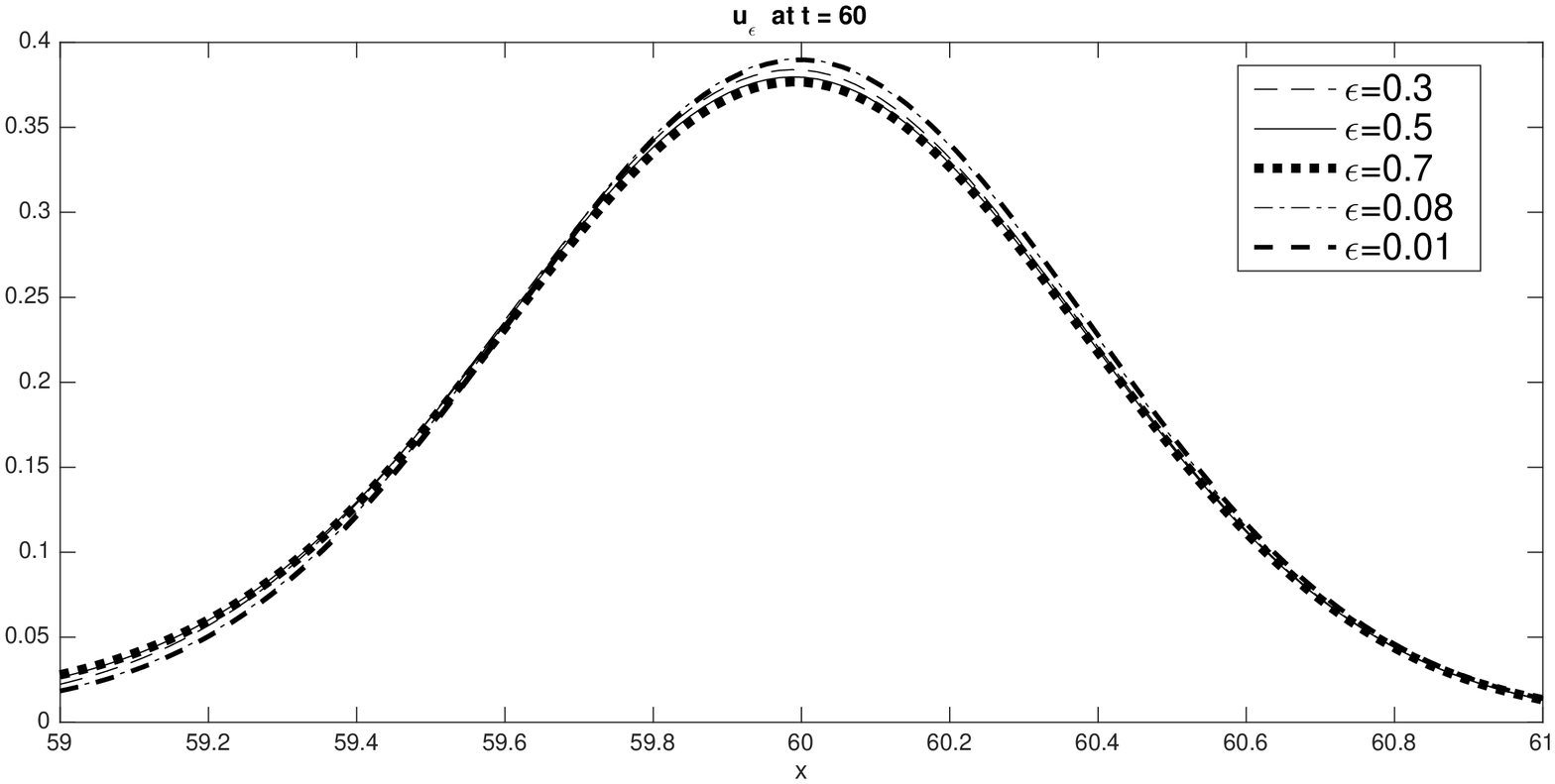}
\caption{Comparison of the solution $u_\epsilon$ at time $t=60$ of the regularised problem \eqref{EQ: CPbb}, for several values of the parameter $\epsilon$.   }
\label{CompUeps}
\end{figure}

\begin{figure}[ht!]
\centering
\includegraphics[width=15cm,height=8cm]{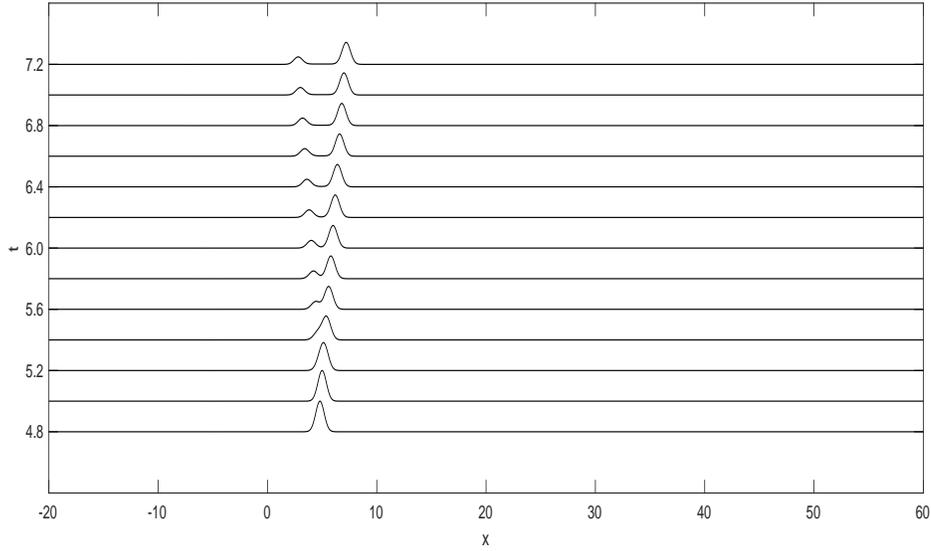}
\caption{In this plot, we can see the evolution of $u_\epsilon$ for $\epsilon=0.01$ at  times $t= 4.8,  5.0, 5.2, 5.4, 5.6, 5.8, 6.0, 6.2 , 6.4, 6.6 , 6.8, 7.0, 7.2$.  }
\label{propagationU2}
\end{figure}

\begin{figure}[ht!]
\centering
\includegraphics[width=15cm,height=8cm]{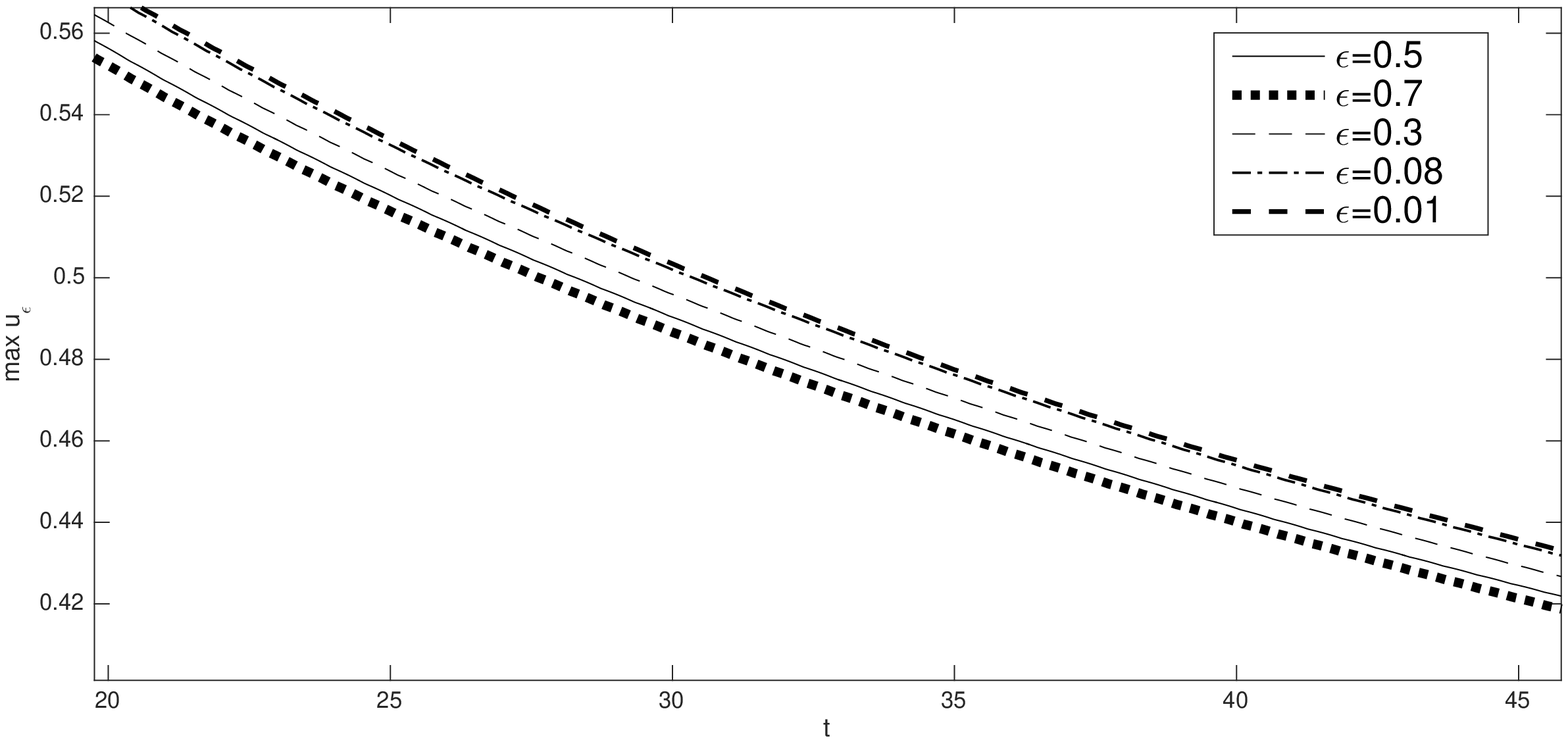}
\caption{In this plot, we can see the decay of the solution $u_\epsilon$ with respect to the time $t$ of the regularised problem \eqref{EQ: CPbb}, for several values of the parameter $\epsilon$.   }
\label{AmplitudeUeps05}
\end{figure}

\begin{figure}[ht!]
\centering
\includegraphics[width=15cm,height=8cm]{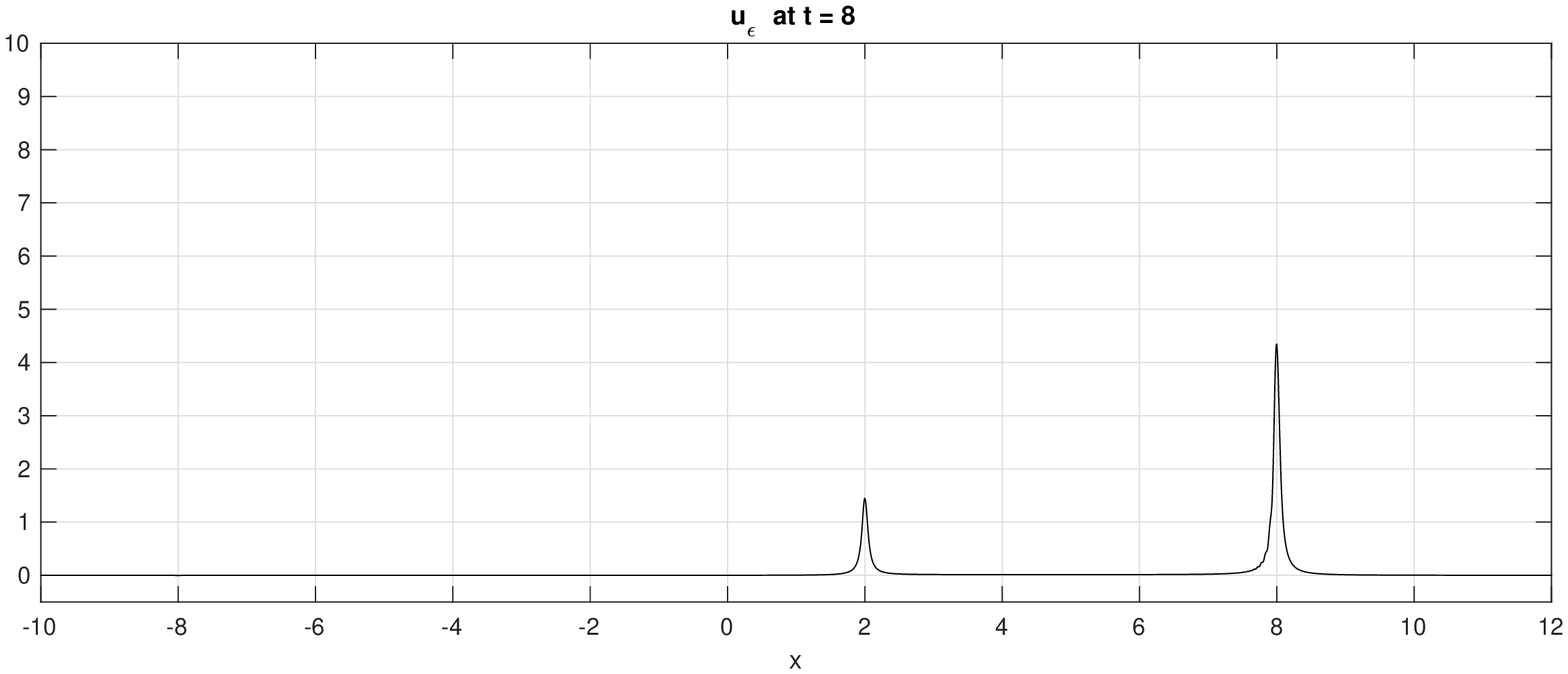}
\caption{In this plot we see the solution $u_\epsilon$ of the regularised problem \eqref{EQ: CPbb} for $\epsilon=0.01$ and initial data \eqref{Cauchydata} with $e=0.05$. }
\label{propagationU3}
\end{figure}

\begin{figure}[ht!]
\centering
\includegraphics[width=15cm,height=8cm]{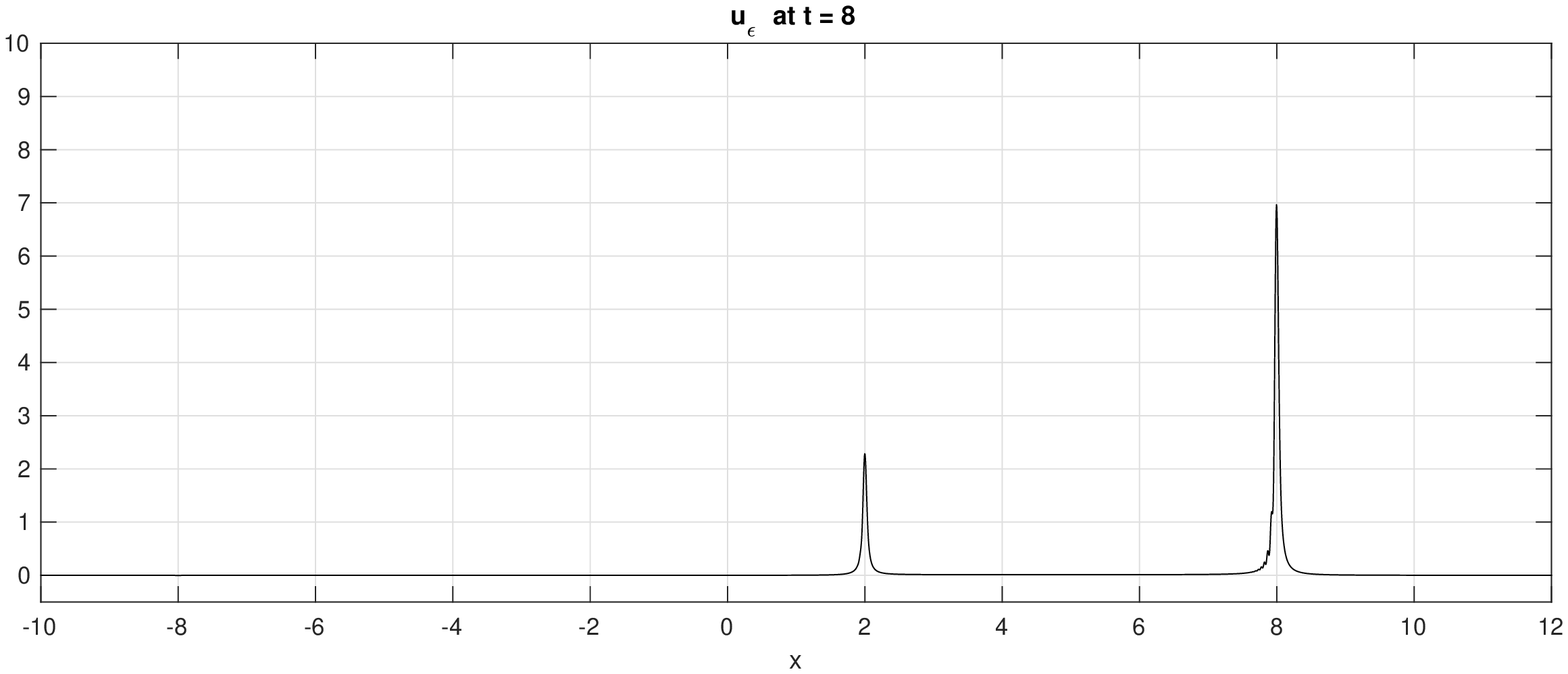}
\caption{In this plot we see the solution $u_\epsilon$ of the regularised problem \eqref{EQ: CPbb} for $\epsilon=0.01$ and initial data \eqref{Cauchydata} with $e=0.03$. }
\label{propagationU4}
\end{figure}

\begin{figure}[ht!]
\centering
\includegraphics[width=15cm,height=8cm]{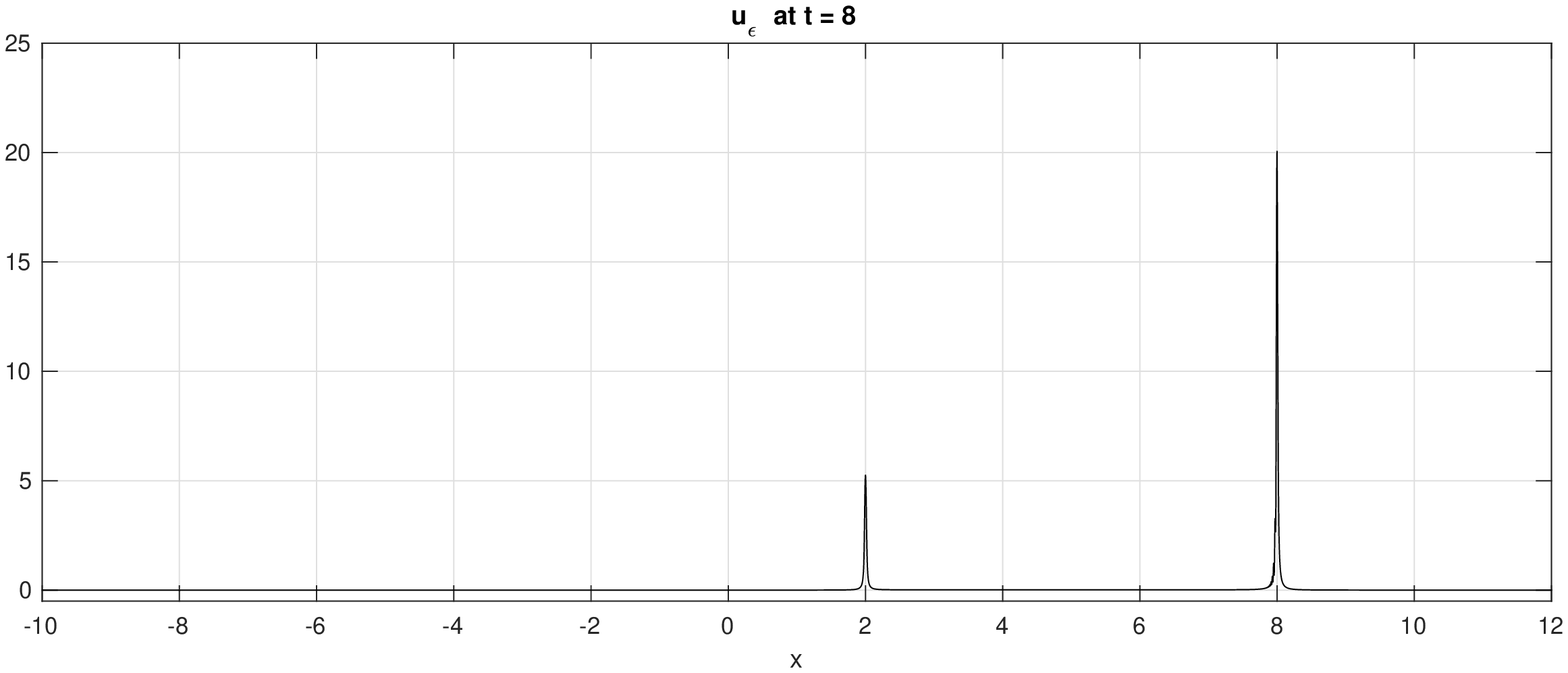}
\caption{In this plot we see the solution $u_\epsilon$ of the regularised problem \eqref{EQ: CPbb} for $\epsilon=0.01$ and initial data \eqref{Cauchydata} with $e=0.01$. }
\label{propagationU5}
\end{figure}

\subsection{Conclusions for the numerical part}

The numerical experiments demonstrate that the approximation techniques work well also in the situation when the strict mathematical formulation of the problem is difficult within the classical theory of distributions. The notion of very weak solutions eliminates this difficulty yielding the well-posedness results for equations with singular coefficients. Within this approach (of very weak solutions) one can recover the expected physical properties of the equation, for example the propagation profile and the decay of the sup-norm of solutions for large times.
Moreover, we seem to discover a new interesting phenomenon, presented in
Figure \ref{propagationU2} and analysed further in Figures \ref{propagationU3}-\ref{propagationU5}: the appearance of a new (reflective) wave (shortly) after the singular time, travelling in the direction opposite to the main one. This can be explained as an {\em echo effect} in the original acoustic problem produced at the interfaces of discontinuity of the medium. Moreover, the reflected wave appears to be of the same regularity as the original one, just smaller in amplitude. Therefore, this phenomenon is different from the one appearing in conical refraction in the presence of multiple characteristics.

\bigskip
Conflict of Interest: The authors declare that they have no conflict of interest.

\end{document}